\documentclass[12pt]{article}
\usepackage{mathrsfs}
\usepackage{amssymb}
\usepackage{}
\usepackage{amsfonts}
\usepackage{amsfonts,amsmath,amsthm, amssymb}
\usepackage{array}
\usepackage{latexsym, euscript, epic, eepic}
\usepackage{tikz}
\usepackage[noblocks]{authblk}
\usetikzlibrary{matrix}
\newtheorem{lemma}{Lemma}[section]
\newtheorem{theorem}{Theorem}[section]
\newtheorem{definition}{Definition}[section]

\newtheorem{remark}{Remark}[section]

\def \R{\mathbb{R}}

\def\R {\Bbb R}

\begin{document}
 
\title{Carleson measures and chord-arc curves \thanks{Research supported by the National Natural Science Foundation of China (Grant Nos. 11501259, 11671175).}}

\author{Huaying Wei \thanks{Department of Mathematics and Statistics, Jiangsu Normal University, Xuzhou 221116, PR China. Email:  6020140058@jsnu.edu.cn.  } ,   Michel Zinsmeister \thanks{MAPMO, Universit\'e d' Orl\'eans, Orl\'eans  Cedex 2, France. Email: zins@univ-orleans.fr}}
\date{}
\maketitle

\begin{center}
\begin{minipage}{120mm}
{\small{\bf Abstract}.
Following Semmes \cite{Se} and Zinsmeister \cite{Zi89}, we continue the study of Carleson measures and their invariance under pull-back and push-forward operators. We also study the analogous statements for vanishing Carleson measures. As an application, we show that some quotient space of the space of chord-arc curves has a natural complex structure.
}

\end{minipage}
\end{center}


 \vskip1cm

\section{Introduction}
A positive measure $\mu$ defined in a simply connected domain $\Omega$ is called a Carleson measure (see \cite{Ga}) if
\begin{equation}\label{carlesonnorm}
\|\mu\|_{*} = \sup \{\frac{\mu(\Omega \cap D(z, r))}{r}: z \in \partial \Omega, 0 < r < diameter(\partial\Omega)\} < \infty,
\end{equation}
where $D(z, r)$ is the disk with center $z$ and radius $r$. A Carleson measure $\mu$ is called a vanishing Carleson measure if $\lim_{r\rightarrow 0}\mu(\Omega \cap D(z, r))/r = 0$ uniformly for $z \in \partial \Omega$. We denote by $CM(\Omega)$ and $CM_0(\Omega)$ the set of all Carleson measures and vanishing Carleson measures on $\Omega$, respectively. It is easy to see that $CM(\Omega)$ is a Banach space with the Carleson norm $\|\cdot\|_{*}$. 

Let $\varphi$ be a conformal mapping from the unit disk $\Delta$ onto a simply connected domain $\Omega$. For any $\mu \in CM(\Omega)$, the pull-back of $\mu$ is the measure defined on $\Delta$ by
$$ \varphi^{*}d\mu=|\varphi^{'}|^{-1}d(\mu\circ \varphi).$$
For a Carleson measure $\nu$ on $\Delta$ we define similarly the push-forward of $\nu$ as being the measure on $\Omega$ defined by 
$$(\varphi^{-1})^{*}d\nu = |(\varphi^{-1})^{'}|^{-1}d(\nu\circ \varphi^{-1}).$$

If $\Omega=\Delta$, these two operators are isomorphisms of $CM(\Delta)$  (one being the reciprocal of the other): this is another way of stating  the conformally invariant character of Carleson measures on $\Delta$ as in \cite[p.231]{Ga}. 

In 1989, Zinsmeister \cite{Zi89} proved the following:
\begin{theorem}
Let $\varphi$ be a conformal mapping from the unit disk $\Delta$ onto a simply connected domain $\Omega$. Then the following two statements hold:
\begin{description}
\item[(Z1)] $\log \varphi^{'} \in BMOA(\Delta)$ if and only if the pull-back operator $\varphi^{*}$ is bounded from $CM(\Omega)$ to $CM(\Delta)$;
\item[(Z2)] If $\partial \Omega$ is Ahlfors-regular, then the push-forward operator $(\varphi^{-1})^{*}$ is  bounded from $CM(\Delta)$ to $CM(\Omega)$.
\end{description}
\end{theorem}
Recall that a curve $\Gamma \subset \mathbb{C}$ is Ahlfors-regular if, $\Lambda^{1}$ denoting the Hausdorff linear measure, 
$$\exists C_1 > 0;\, \forall z \in \mathbb{C},\, \forall r > 0, \,\Lambda^{1}(\Gamma \cap D(z, r)) \leqslant C_1r.$$

It should be pointed out that if $\mu \in CM(\Omega)$ is absolutely continuous (with respect to Lebesgue measure), that is,  if  there exists a function $\lambda \in L^{1}$ such that 
$$d\mu(z) = \lambda(z)dxdy,$$
then, writing $d\nu=\varphi^*d\mu$, $$d\nu(\zeta) = \lambda\circ\varphi(\zeta) |\varphi^{'}(\zeta)|d\xi d\eta.$$ 
In 1988, Semmes \cite{Se} proved the following: 
\begin{theorem}\label{Semmesthm}
Let $\varphi$ be a quasiconformal mapping of $\Delta$ onto $\Delta$ that satisfies
\begin{description}
\item[(S1)] $\varphi$ is bi-Lipschitz continuous under the Poincar\'e metric,
\item[(S2)] $\varphi|_{\mathbb{S}}$ is a strongly quasisymmetric homeomorphism.
\end{description}
If $\lambda(z)dxdy \in CM(\Delta)$, then $\lambda \circ \varphi(\zeta) |\partial \varphi| d\xi d\eta \in CM(\Delta)$, with norm dominated by the norm of $\lambda(z) dxdy$. 
\end{theorem}
Furthermore, in \cite{TWS} it is shown that $\lambda \circ \varphi (\zeta) |\partial \varphi| d\xi d\eta \in CM_0(\Delta)$ if $\lambda(z)dxdy \in CM_0(\Delta)$.

In Section 3,  following Semmes \cite{Se} and Zinsmeister \cite{Zi89}, we continue the study of Carleson measures and their invariance under pull-back and push-forward operators. We also study the analogous statements for vanishing Carleson measures.
In Section 4, as an application of Carleson measure theory, we study further the space of chord-arc curves in the framework of the theory of BMO-Teichm\"uller spaces: in particular, we show that some quotient space of the space of chord-arc curves has a natural complex structure. Before stating these results, we recall the standard theory of Teichm\"uller spaces in Section 2.

\section{Teichm\"uller theory}
In our case, it is convenient to consider the universal Teichm\"uller space $T$, which is identified with the group $QS$ of quasisymmetric automorphisms of the unit circle 
$\mathbb{S}$ modulo post-composition of M\"obius transformations $M\ddot{o}b(\mathbb{S})$, namely, $T = M\ddot{o}b(\mathbb{S})\backslash QS$.  Here a sense preserving 
self-homeomorphism $h$ of the unit circle $\mathbb{S}$ is quasisymmetric if there exists some $M > 0$ such that 
$$
\frac{1}{M} \leqslant \frac{|h(e^{i(\theta + t)}) - h(e^{i\theta})|}{|h(e^{i\theta}) - h(e^{i(\theta - t)})|} \leqslant M
$$
for all $\theta$ and $t > 0$. Let $B(\Delta^{*})$ denote the Banach space of functions $\phi$ holomorphic in the exterior of the unit disk $\Delta^{*}$ with norm 
$$\|\phi\|_{B} = \sup_{z \in \Delta^{*}} (|z|^2 - 1)^2|\phi(z)|.$$
$B_0(\Delta^{*})$ is the subspace of $B(\Delta^{*})$ consisting of all functions $\phi$ such that $(|z|^2 - 1)^2|\phi(z)| \to 0$ as $|z| \to 1^{+}$. 

The Bers embedding $\varPhi$ of $T$ is a homeomorphism of $T$ onto a bounded domain in $B(\Delta^{*})$. The full definition of $\varPhi$ involves several steps, which we list here:\\
1. Select a representative $h$ of an element $[h]$ in $T = M\ddot{o}b(\mathbb{S})\backslash QS$ ($h$ is a quasisymmetric homeomorphism of $\mathbb{S}$),\\
2. take any quasiconformal self-map $\tilde{h}$ of $\Delta$ such that $\tilde{h}$ is an extension of $h$,\\
3. form the Beltrami coefficient $\mu$ of $\tilde{h}$, namely, $\mu(z) = \tilde{h}_{\bar{z}}/\tilde{h}_{z}$,\\
4. let $\tilde{\mu}(z) = \mu(z)$ for $z \in \Delta$ and $\tilde{\mu}(z) = 0$ for $z \in \Delta^{*}$ and solve for $f$ in the Beltrami equation 
$$f_{\bar{z}}(z) = \mu(z) f_{z}(z)$$
to obtain a quasiconformal homeomorphism of $\hat{\mathbb{C}}$ holomorphic in $\Delta^{*}$,\\
5. take the Schwarzian derivative $\mathcal{S}(f)$ of $f$ in $\Delta^{*}$.\\
The Bers embedding is the map $[h] \mapsto \varPhi(h) = \mathcal{S}(f)$ (see \cite{Le}). 

Via the Bers embedding, $T$ carries a natural complex Banach manifold structure modeled on the Banach space $B(\Delta^{*})$. Recall that the small Teichm\"uller space, $T_0 = M\ddot{o}b(\mathbb{S})\backslash Sym$, as an important subspace of the universal Teichm\"uller space, has been introduced and well studied by Gardiner and Sullivan \cite{GS} in 1992. Here the subgroup 
$Sym \subset QS$ consists of symmetric automorphisms of $\mathbb{S}$.  Recall that a quasisymmetric automorphism $h$ is said to be symmetric if 
$$
\lim_{t \to 0^+}\frac{|h(e^{i(\theta + t)}) - h(e^{i\theta})|}{|h(e^{i\theta}) - h(e^{i(\theta - t)})|} = 1
$$
uniformly for all $\theta$. The mapping $\varPhi$ described above  applied to $T_0$ has image in $B_0(\Delta^{*})$. The Bers embedding also provides a natural way to make $T_0$ into complex manifold modeled on the Banach space $B_0(\Delta^{*})$. 

Furthermore, let $\hat{\varPhi}$ induced by the Bers embedding $\varPhi$ be the map from $Sym\backslash QS$ into $B_0\backslash B$. $\hat{\varPhi}$ is defined in exactly the same way as $\varPhi$ with the exception that $\hat{\varPhi}$ is viewed as defined on the right cosets of $Sym$ in $QS$ with image in $B_0\backslash B$. 

In 1992, Gardiner and Sullivan \cite{GS} proved that the map $\hat{\varPhi}$ from $Sym\backslash QS$ into $B_0\backslash B$ is well-defined and locally one-to-one.  
We have the commutative diagram:

\begin{tikzpicture}
\matrix(m)[matrix of math nodes, row sep=5em, column sep=5em]
{ |[name=ka]| M\ddot{o}b(\mathbb{S})\backslash QS & |[name=kb]| B\\
|[name=kc]| Sym\backslash QS & |[name=kd]| B_0\backslash B\\};

\draw[->, thick] (ka) edge node[auto] {$\varPhi$} (kb)
                         (ka) edge node[auto] {$\pi$} (kc)
                         (kc) edge node[auto] {$\hat{\varPhi}$} (kd)
                         (kb) edge node[auto] {$p$} (kd)
                         ;
\end{tikzpicture}

Thus, the map $\hat{\varPhi}$ yields local coordinates for $Sym\backslash QS$ in the Banach quotient space $B_0\backslash B$, and then the coset space 
$Sym\backslash QS$ 
becomes a complex  manifold modeled in the Banach space $B_0\backslash B$. The following result \cite[Section 16.8]{GL}, which implies that $\hat{\varPhi}$ is a global coordinate,  was first observed by Jeremy Kahn. 
\begin{theorem}\label{2.1}
The map $\hat{\varPhi}$ from $Sym\backslash QS$ into $B_0\backslash B$ is an isomorphism.
\end{theorem}
Let $\Omega$ be a simply connected domain in $\hat{\mathbb{C}}$ bounded by a Jordan curve $\Gamma$. $\Omega$ and the complement $\Omega^{*}$ of $\Omega \cup \Gamma$ are called complementary Jordan domains. The conformal maps $f$ and $g$ mapping $\Delta^{*}$ onto $\Omega^{*}$ and $\Omega$ onto $\Delta$ extend continuously to $\Gamma$, and thus the composition $g\circ f$ restricted to the unit circle $\mathbb{S}$ is a homeomorphism $h$. We call $h$ the welding homeomorphism corresponding to $\Gamma$.  We denote by $SQS$ the set of strongly quasisymmetric homeomorphisms $h$ on the unit circle $\mathbb{S}$ which are welding homeomorphisms corresponding to the quasicircles satisfying the Bishop-Jones condition (called BJ quasicircles) (see \cite{BJ}). In other words, $h$ is strongly quasisymmetric if and only if it is absolutely continuous with density $h^{'}$ belonging to the class of weights $A_{\infty}$ (see \cite{Ga}) introduced by Muckenhoupt, in particular, $\log h^{'} \in BMO(\mathbb{S})$ and
$$d([h_1], [h_2]) = \|\log h_2^{'} - \log h_{1}^{'}\|_{BMO},\;\;\; [h_1],[h_2] \in M\ddot{o}b(\mathbb{S})\backslash SQS$$
defines a topology in $M\ddot{o}b(\mathbb{S})\backslash SQS$.
Let $LQS$ consist of welding homeomorphisms corresponding to chord-arc curves. 
\begin{definition}
We call a curve $\Gamma$ is  a chord-arc curve (also called Lavrentiev curve) with constant $C_2$, if
$$\Lambda^{1}(\widetilde{\zeta z}) \leqslant C_2|\zeta - z|$$
for the smaller subarc $\widetilde{\zeta z}$ of $\Gamma$ joining any two finite points $z$ and $\zeta$ of $\Gamma$.  A domain bounded by a chord-arc curve with constant 
$C_2$ is called a $C_2$-chord-arc domain.
\end{definition}
Let $SS$ be the set of strongly symmetric homeomorphisms $h$ on $\mathbb{S}$ which are absolutely continuous, with $\log h^{'} \in VMO(\mathbb{S})$. In other words, $SS$ is the set of welding homeomorphisms corresponding to asymptotically smooth curves in the sense of Pommerenke \cite[p.172]{Po92}, which satisfy
$$\Lambda^{1}(\widetilde{\zeta z})/|\zeta - z| \to 1,\; as\, |\zeta - z| \to 0, \;\zeta, \, z \in \Gamma.$$
It is clear that we have the increasing scale of sets $SS \subset LQS \subset SQS$.

We denote by $\mathcal{B}(\Delta^{*})$ the Banach space of functions $\phi$ holomorphic in $\Delta^{*}$ each of which induces a Carleson measure $\lambda_{\phi}$ by 
$d\lambda_{\phi}(z) = |\phi(z)|^2(|z|^2 - 1)^3 dxdy \in CM(\Delta^{*})$. The norm on $\mathcal{B}(\Delta^{*})$ is
$$\|\phi\|_{\mathcal{B}} = \|\lambda_{\phi}\|_{*}.$$
\cite[Lemma 4.1]{SW} implies $\mathcal{B}(\Delta^{*}) \subset B(\Delta^{*})$, and the inclusion map is continuous. We denote by $\mathcal{B}_0(\Delta^{*})$ the subspace of 
$\mathcal{B}(\Delta^{*})$ consisting of all functions $\phi$ such that $\lambda_{\phi} \in CM_0(\Delta^{*})$. Then $\mathcal{B}_0(\Delta^{*}) \subset B_0(\Delta^{*})$.

We  claim that all above properties of the coset space $Sym\backslash QS$
 carry over if one view $\hat{\varPhi}$ as a mapping from the space $SS\backslash LQS$ into the Banach quotient space $\mathcal{B}_0\backslash \mathcal{B}$. For more information about the map 
$\hat{\varPhi}$, we refer the readers to \cite{Ma}.
\begin{theorem}\label{main}
The map $\hat{\varPhi}$ from $SS\backslash LQS$ onto its image in $\mathcal{B}_0\backslash \mathcal{B}$ is well-defined and globally one-to-one. Consequently, the coset space $SS\backslash LQS$ becomes a complex  manifold modeled on the Banach space $\mathcal{B}_0\backslash \mathcal{B}$.
\end{theorem}
We hope that this complex analytic theory could find applications to some other problems in the study of chord-arc curves. In Section 4, we will give the proof of Theorem \ref{main}.

\section{On Carleson measures}


The following result by Bishop and Jones \cite{BJ} gives a geometric characterization of BMOA domain.
\begin{lemma}\label{BJ}
Let $\varphi$ be conformal on $\Delta$.  Then $\log \varphi^{'} \in BMOA(\Delta)$ if and only if the domain $\Omega = \varphi(\Delta)$ satisfies the following Bishop-Jones (BJ) condition:\\
 For any $z\in\Omega$ there exists a $k(\Omega)$-chord-arc domain $\Omega_z\subset\Omega$ containing $z$, whose diameter is uniformly comparable to $dist(z,\partial\Omega)$, and such that
 $$ \Lambda^1(\partial\Omega\cap \partial\Omega_z)\geq c(\Omega)dist(z,\partial\Omega),$$
 where $k(\Omega)>1$ and $c(\Omega)>0$ depend only on $\Omega$. 
\end{lemma}

We next state a well-known corollary of Koebe distortion theorem (see \cite[p.9]{Po92}).  For a conformal map of $\Delta$,  define
$$d_f(z) = dist(f(z), \partial f(\Delta)) \;\; \;\;for \; z \in \Delta. $$
\begin{lemma}\label{Koebe}
If $f$ maps $\Delta$ conformally into $\mathbb{C}$ then 
$$\frac{1}{4}(1 - |z|^2)|f^{'}(z)| \leqslant d_f(z) \leqslant (1 - |z|^2)|f^{'}(z)| \;\;\;\; for\; z \in \Delta. $$
\end{lemma}

Now we prove the analogous statement of (Z1) for vanishing Carleson measures.
\begin{theorem}\label{Delta}
Let $\varphi$ be conformal on $\Delta$ and $\Omega = \varphi(\Delta)$ be a BJ quasidisk. Then the pull-back operator
\begin{equation*}
\varphi^{*}: \; CM_0(\Omega) \to CM_0(\Delta)
\end{equation*}
is well-defined and bounded.
\end{theorem}
\begin{proof}
Suppose $\varphi$ be conformal from $\Delta$ onto BJ quasidisk $\Omega$. Then $\varphi$ can be extended  to a quasiconformal homeomorphism in $\hat{\mathbb{C}}$ (still denoted by $\varphi$). Thus, $\varphi$ is bi-H\"older in $\overline{\Delta}$. That is, 
\begin{equation}\label{bi}
C_1|z_1 - z_2|^{1/\alpha} \leqslant |\varphi(z_1) - \varphi(z_2)| \leqslant C_2|z_1 - z_2|^{\alpha}, \;\;\; z_1, z_2 \in \overline{\Delta},
\end{equation}
here the constants $C_1$, $C_2$ and $\alpha$ depend only on the conformal mapping $\varphi$. Let $\mu \in CM_0(\Omega)$ and $d\nu = \varphi^{*}d\mu = |\varphi^{'}|^{-1} d(\mu\circ\varphi)$. Then for any $\epsilon > 0$, there exists a constant $r_0 > 0$ such that 
$\frac{1}{r}\mu(D(w, r)\cap \Omega) \leqslant \epsilon$ 
uniformly for $w \in \Gamma$ when $0 < r \leqslant r_0$. Denote by $\varphi^{-1}(w) = z$. It follows from (\ref{bi}) that there exists a constant $\lambda_0 > 0$ such that 
\begin{equation}\label{D}
\varphi(D(z, \lambda_0)\cap\Delta) \subset D(w, r_0)\cap\Omega.
\end{equation}
Let $d\mu^{'} = d(\mu\chi_{D(w, r_0)\cap\Omega})$ and $d\nu^{'} = \varphi^{*}d\mu^{'} = |\varphi^{'}|^{-1} d(\mu^{'}\circ\varphi)$ . 
Here $\chi_{D(w, r_0)\cap\Omega}$ denotes the characteristic function of the intersection $D(w, r_0)\cap\Omega$. Then 
$\|\mu^{'}\|_{*} \leqslant \epsilon$. We conclude from Lemma \ref{BJ} and (Z1) that 
$\nu^{'}$ is a Carleson measure with norm dominated by $\epsilon$.  By means of (\ref{D}) we have
$d\tilde{\nu} = |\varphi^{'}|^{-1} d(\mu\circ\varphi\chi_{D(z, \lambda_0)\cap\Delta}) \leqslant d\nu^{'}$. 
Then $\tilde{\nu}$ is a Carleson measure with norm dominated by $\epsilon$. Thus, $\nu  \in CM_0(\Delta)$.
\end{proof}

Let $\Omega$ be a domain bounded by the Ahlfors-regular curve $\Gamma$ with constant $C_1$. For any small constant $r > 0$,  
let $\Omega_r = \{z \in \Omega; \, dist(z, \Gamma) >  r\}$. Denote by $\partial \Omega_r = \Gamma_r$.

\begin{lemma}\label{0}
$\mu \in CM(\Omega)$ is a vanishing Carleson measure in $\Omega$ if and only if $\|\mu - \mu_r\|_{*} \to 0 $ as $r \to 0^{+}$.  Here $d\mu_r = d(\mu \chi_{\Omega_r})$ and 
$\chi_{\Omega_r}$ is the characteristic function of the domain $\Omega_r$.
\end{lemma}

\begin{proof}
Suppose $\|\mu - \mu_r\|_{*} \to 0$ as $r \to 0^{+}$. Then for any $\epsilon > 0$, there exists a constant $r_0 > 0$  such that $\|\mu - \mu_{r_0}\|_{*} < \epsilon$. We conclude that 
when $0 < h < diameter(\partial \Omega)$, for any $z \in \Gamma$, 
we have $(\mu - \mu_{r_0}) (\Omega \cap D(z, h)) < \epsilon h$ 
which implies that 
$$\mu(\Omega \cap D(z, h)) =  (\mu - \mu_{r_0}) (\Omega \cap D(z, h)) < \epsilon h$$ uniformly for $z \in \Gamma$  when  $0 < h < r_0.$ Thus,  $\mu \in CM_0(\Omega)$.

Conversely, suppose $\mu \in CM_0(\Omega)$. For any $\epsilon > 0$, there exists a constant $h_0 > 0$ such that
$\mu(\Omega \cap D(z, h)) < \epsilon h$ uniformly for $z \in \Gamma$ when $0 < h \leqslant h_0$. We choose $r_0 = \frac{1}{2}h_0$. Then when $0 < h \leqslant h_0$, 
\begin{equation}\label{small}
(\mu - \mu_{r_0})(\Omega \cap D(z, h)) \leqslant \mu(\Omega \cap D(z, h)) < \epsilon h,
\end{equation}
uniformly for $z \in \Gamma$. It remains to show that $(\mu - \mu_{r_0})(\Omega \cap D(z, h)) < \epsilon h$ uniformly for $z \in \Gamma$ as $h_0 < h < diameter(\partial\Omega)$. 
For any $z \in \Gamma$,  assume $\zeta, w$ be respectively the first and the last (for an orientation of the curve) points of $\Gamma \cap \partial D(z, h)$. Suppose  
$\zeta_0 = z$, $ \zeta_1, \zeta_2, \cdot\cdot\cdot, \zeta_{n - 1} \in \widetilde{\zeta z}$ and $\zeta_n \in \Gamma$ satisfy  $|\widetilde{\zeta_n z}| \geqslant  |\widetilde{\zeta z}|$ and 
$|\zeta_{i + 1} \zeta_{i}| = h_0, \, i = 0, 1, \cdot\cdot\cdot, n - 1$. Similarly, suppose 
$w_0 = z$, $ w_1, w_2, \cdot\cdot\cdot, w_{m - 1} \in \widetilde{zw}$ and $w_m \in \Gamma$ satisfy  $|\widetilde{zw_m}| \geqslant  |\widetilde{zw}|$ and 
$|w_{j} w_{j + 1}| = h_0, \, j = 0, 1, \cdot\cdot\cdot, m - 1$. 
Then
\begin{equation}\label{big}
\begin{split}
(\mu - \mu_{r_0})(\Omega \cap D(z, h)) & \leqslant  \sum_{i = 0}^{n} \mu(\Omega \cap D(\zeta_i, h_0)) + \sum_{j = 0}^{m} \mu(\Omega \cap D(w_j, h_0))\\
& \leqslant (n + 1) h_0 \epsilon + (m + 1) h_0 \epsilon\\
& \leqslant 4 h_0 \epsilon + |\widetilde{\zeta w}| \epsilon\\
& \leqslant 4 h_0 \epsilon +  C_1 h \epsilon\\
& < (4 + C_1)h\epsilon.\\
\end{split}
\end{equation}
Combining (\ref{small}) and (\ref{big}), we conclude that for $0 < r \leqslant r_0$, 
$$(\mu - \mu_{r})(\Omega \cap D(z, h)) \leqslant (\mu - \mu_{r_0})(\Omega \cap D(z, h)) <  (4 + C_1)h\epsilon $$
uniformly for $z \in \Gamma$ when $0 < h < diameter (\partial\Omega)$. Consequently, $\|\mu - \mu_r\|_{*} \to 0 $ as $r \to 0^{+}$.

\end{proof}


Based on Lemma \ref{0} , we show the analogous statement of (Z2) for vanishing Carleson measures is still valid.
\begin{theorem}\label{Omega}
Let $\Omega$ be an Ahlfors-regular domain and $\varphi$ map $\Delta$ conformally onto $\Omega$. 
Then the push-forward operator 
$$(\varphi^{-1})^{*}: \; CM_0(\Delta) \to CM_0(\Omega)$$ is well-defined and bounded.
\end{theorem}
\begin{remark}
With the same hypothesis the pull-back operator
$$\varphi^{*}: \; CM_0(\Omega) \to CM_0(\Delta)$$ is also well-defined and bounded; the (similar) proof is left to the reader. 
\end{remark}
\begin{remark}
For the convenience, the measure and its density will be identified in the following arguments.
\end{remark}
\begin{proof}
 Suppose $\nu \in CM_0(\Delta)$. Then for any $\epsilon > 0$, there exists 
$r_0 \in (0, 1)$ such that 
$\|\nu - \nu\chi_{\Delta_{r_{0}}}\|_{*} < \epsilon.$ 
We conclude by Lemma \ref{Koebe} that there exists a small constant $r_0^{'}$ such that $\varphi(\Delta_{r_0}) \subset \Omega_{r_0^{'}}$ which implies 
$(\nu\chi_{\Delta_{r_0}})\circ \varphi^{-1}(A) \leqslant \nu \circ (\varphi^{-1}\chi_{\Omega_{r_0^{'}}})(A)$ for any $A \subset \Omega$. 
Combining the last inequality and (Z2), we obtain 
\begin{equation*}
\begin{split}
\|(\varphi^{-1})^{*}\nu -  (\varphi^{-1})^{*}\nu\chi_{\Omega_{r^{'}}}\|_{*}
 & \leqslant \|(\varphi^{-1})^{*}\nu -  (\varphi^{-1})^{*}\nu\chi_{\Omega_{r_0^{'}}}\|_{*} 
 \leqslant \|(\varphi^{-1})^{*}\nu -  (\varphi^{-1})^{*}(\nu\chi_{\Delta_{r_0}})\|_{*} \\
& \leqslant C\|\nu - \nu\chi_{\Delta_{r_{0}}} \|_{*}
 < C\epsilon\\
\end{split}
\end{equation*}
when $0 < r^{'} \leqslant r_0^{'}$. Here the constant $C$ is the norm of the operator $(\varphi^{-1})^{*}$. Consequently, $(\varphi^{-1})^{*}\nu \in CM_0(\Omega).$
\end{proof}

\begin{definition}[see \cite{Po92} p.168]
Let $\omega(z) > 0$ be locally integrable on the unit circle $\mathbb{S}$. Set $\omega(E) = \int_E \omega(z) |dz|$. Denote by $|E|$ the Lebesgue measure of $E$. We say that 
$\omega$ satisfies the Coifman-Fefferman ($A_{\infty}$) condition if one of the following two equivalent conditions holds:
\begin{description}
\item[(1a)] there exist $C_1 > 0, C_2 > 0$ such that 
$$\frac{\omega(E)}{\omega(I)} \leqslant C_2\Big(\frac{|E|}{|I|}\Big)^{C_1} $$
for all  subarcs $I \subset \mathbb{S}$ and measurable sets $E \subset I$.\\
\item[(2a)] There exists $\beta > 0$ such that for all subarcs $I \subset \mathbb{S}$ and measurable sets $E \subset I$, 
$$\frac{\omega(E)}{\omega(I)}  < \beta \Longrightarrow \frac{|E|}{|I|} < 1/2.$$\\
\end{description}
\end{definition}
The following Lemma \cite[p.169]{Po92} states an important property of chord-arc domains:
\begin{lemma}\label{A}
Let $f$ map $\Delta$ conformally onto $\Omega$: If $\Omega$ is a chord-arc domain then $|f^{'}|$ satisfies the $(A_{\infty})$ condition.
\end{lemma}

\begin{definition}[see \cite{Se}]
Suppose $\gamma(z)$ maps $\mathbb{S}$ homeomorphically onto $\Gamma$. We say that $\gamma(z)$ is a strongly quasisymmetric homeomorphism, if
\begin{description}
\item[(1b)] it is locally absolutely continuous,
\item[(2b)] $\Gamma$ is a chord-arc curve,
\item[(3b)] $|\gamma^{'}(z)| \in A_{\infty}.$
\end{description}
\end{definition}

Now we give the invariance of Carleson measures and vanishing Carleson measures under pull-back and push-forward operators induced by quasiconformal mappings 
satisfying certain conditions which generalizes Semmes' result (see Theorem \ref{Semmesthm}).
\begin{theorem}\label{1c}
Let $\Omega$ be a chord-arc domain. If $\varphi$ is a quasiconformal mapping of $\Delta$ onto $\Omega$ that satisfies
\begin{description}
\item[(S1)] $\varphi$ is bi-Lipschitz continuous under the Poincar\'e metric,
\item[(S2)] $\varphi|_{\mathbb{S}}$ is a strongly quasisymmetric homeomorphism,
\end{description}
then $\lambda \circ \varphi^{-1}(z) |\partial \varphi^{-1}|dxdy \in CM(\Omega)$ when $\lambda(z) dxdy \in CM(\Delta)$. Furthermore,  $\lambda \circ \varphi^{-1}(z) |\partial \varphi^{-1}|dxdy \in CM_0(\Omega)$ when $\lambda(z) dxdy \in CM_0(\Delta)$.
\end{theorem}

\begin{proof}
Let $f: \Omega \to \Delta$ be conformal and the map $\psi: \Delta \to \Delta$ satisfy $\psi \circ f = \varphi^{-1}$. That is, $\psi = \varphi^{-1}\circ f^{-1}$.

Denote by $\rho_{\Omega}$ the Poincar\'e metric in $\Omega$, that is, $\rho_{\Omega}(z)|d z| = |d f(z)| /(1 - |f(z)|^2)$ for $z \in \Omega$. The condition (S1) implies there 
exists a constant $C > 1$ such that 
$$\frac{1}{C}\frac{|d z|}{1 - |z|^2} \leqslant \rho_{\Omega}(\varphi(z))|d \varphi(z)| \leqslant C \frac{|d z|}{1 - |z|^2}.$$
We conclude that 
 $$\frac{1}{C}\frac{|d z|}{1 - |z|^2} \leqslant \frac{|d (f\circ \varphi (z))|}{1 - |f \circ \varphi(z)|^2} \leqslant C \frac{|d z|}{1 - |z|^2}.$$
 Consequently, 
$$\frac{1}{C}\frac{|d z|}{1 - |z|^2} \leqslant \frac{|d \psi^{-1} (z)|}{1 - |\psi^{-1}(z)|^2} \leqslant C \frac{|d z|}{1 - |z|^2}$$
from which we obtain $\psi$ is bi-Lipschitz under the Poincar\'e metric.

Combining $(2a)$ and (S2), there exists $\epsilon > 0$ such that for all subarcs $I \subset \mathbb{S}$ and measurable sets $E \subset I$, 
$$\frac{|\varphi(E)|}{|\varphi(I)|}  < \epsilon \Longrightarrow \frac{|E|}{|I|} < 1/2.$$
On the other hand, By $(1a)$ and Lemma \ref{A}, we conclude that  for above $\epsilon > 0$, there exists $\beta > 0$ such that 
$$\frac{|f \circ \varphi(E)|}{|f \circ \varphi(I)|}  < \beta \Longrightarrow  \frac{|f^{-1}\circ (f\circ \varphi (E))|}{|f^{-1}\circ (f\circ \varphi (I))|} = 
\frac{|\varphi(E)|}{|\varphi(I)|} < \epsilon. $$
We conclude that 
$$\frac{|\psi^{-1}(E)|}{|\psi^{-1}(I)|} < \beta \Longrightarrow \frac{|E|}{|I|} < 1/2. $$
Thus by means of $(2a)$ we have $(\psi^{-1})^{'} \in A_{\infty}$ which implies $\psi^{'} \in A_{\infty}$.  Then $\psi$ is a strongly quasisymmetric homeomorphism.

By Semmes' result,  $\lambda \circ \psi(z) |\partial \psi| dxdy \in CM(\Delta)$ when $\lambda(z) dxdy \in CM(\Delta)$. It follows from (Z2) that
\begin{equation*}
\begin{split}
\lambda \circ \varphi^{-1}(z) |\partial \varphi^{-1}|dxdy & = \lambda \circ (\psi \circ f)(z)|\partial(\psi \circ f)|dxdy\\
& = (\lambda \circ \psi |\partial \psi|) \circ f(z) |f^{'}|dxdy \in CM(\Omega).
\end{split}
\end{equation*}
Furthermore, by  Theorem \ref{Omega}, $\lambda \circ \varphi^{-1}(z) |\partial \varphi^{-1}|dxdy \in CM_0(\Omega)$ when $\lambda(z) dxdy \in CM_0(\Delta)$.

\end{proof}

\begin{theorem}\label{2c}
Let $\Omega$ be a chord-arc domain. If $\varphi$ is a quasiconformal mapping of $\Delta$ onto $\Omega$ that satisfies
\begin{description}
\item[(S1)] $\varphi$ is bi-Lipschitz continuous under the Poincar\'e metric,
\item[(S2)] $\varphi|_{\mathbb{S}}$ is a strongly quasisymmetric homeomorphism,
\end{description}
then $\alpha \circ \varphi(z) |\partial \varphi| dxdy \in CM(\Delta)$ when $\alpha(z) dxdy \in CM(\Omega)$. Furthermore,  $\alpha \circ \varphi(z) |\partial \varphi| dxdy \in CM_0(\Delta)$ when $\alpha(z) dxdy \in CM_0(\Omega)$.
\end{theorem}
\begin{proof}
Let $f: \Delta \to \Omega$ be conformal and $\psi = f^{-1}\circ \varphi$. 

The proof of Theorem \ref{1c} implies $\psi$ is bi-Lipschitz continuous under the Poincar\'e metric and $\psi|_{\mathbb{S}}$ is a strongly quasisymmetric homeomorphism.  
By (Z1), 
$\alpha \circ f(z) |f^{'}| dxdy \in CM(\Delta)$ when $\alpha(z) dxdy \in CM(\Omega)$. On the other hand, By Semmes' result, 
\begin{equation*}
\begin{split}
\alpha \circ \varphi(z) |\partial \varphi|dxdy & = \alpha \circ (f \circ \psi)(z)|\partial(f \circ \psi)|dxdy\\
& = (\alpha \circ f |f^{'}|) \circ \psi(z) |\partial \psi|dxdy \in CM(\Delta).\\
\end{split}
\end{equation*}
Furthermore, by  Theorem \ref{Delta},  $\alpha \circ \varphi(z) |\partial \varphi| dxdy \in CM_0(\Delta)$ when $\alpha(z) dxdy \in CM_0(\Omega)$.
\end{proof}

\begin{remark}
In Theorem \ref{1c} and Theorem \ref{2c}, the unit disk $\Delta$ can be replaced by the chord-arc domain. We omit the detail here.
\end{remark}

\section{The quotient space of chord-arc curves}
Let $M(\Delta)$ denote the open unit ball of the Banach space $L^{\infty}(\Delta)$ of essentially bounded measurable functions on $\Delta$. 
We denote by $\mathcal{L}(\Delta)$ the Banach space of all essentially bounded measurable functions $\mu$ on $\Delta$ each of which induces a Carleson measure 
$\lambda_{\mu} \in CM(\Delta)$ by $d\lambda_{\mu}(z) = |\mu(z)|^2/(1 - |z|^2)dxdy$. The norm on $\mathcal{L}(\Delta)$ is defined as
\begin{equation*}
\|\mu\|_{c} = \|\mu\|_{\infty} + \|\lambda_{\mu}\|_{*}.
\end{equation*}
 $\mathcal{L}_0(\Delta)$ is the subspace of $\mathcal{L}(\Delta)$ consisting of all elements $\mu$ such that $\lambda_{\mu} \in CM_0(\Delta)$. 
Set $\mathcal{M}(\Delta) =  M(\Delta) \cap \mathcal{L}(\Delta)$ and $\mathcal{M}_0(\Delta) =  M(\Delta) \cap \mathcal{L}_0(\Delta)$.

Let $\Omega$ be a simply connected domain in $\hat{\mathbb{C}}$ bounded by a quasicircle $\Gamma$. The conformal maps $f$ and $g$ mapping $\Delta^{*}$ onto 
$\Omega^{*}$ and $\Omega$ onto $\Delta$ extend continuously to $\Gamma$ and $h = g\circ f$ is the quasisymmetric welding homeomorphism. By results of Astala-Zinsmeister \cite{AZ}, Bishop-Jones \cite{BJ} and Fefferman-Kenig-Pipher \cite{FKP}, the following three conditions are equivalent: $(B_1) f$ has a quasiconformal extension to $\hat{\mathbb{C}}$ whose complex dilatation $\mu \in \mathcal{M}(\Delta)$ (for instance the Douady-Earle extension, see \cite{CZ}); $(B_2) h \in SQS;$ $(B_3) \mathcal{S}(f) \in \mathcal{B}(\Delta^{*}).$
Furthermore, Pommerenke \cite{Po78} and Shen-Wei \cite{SW} obtained: The following statements are equivalent: $(V_1) f$ has a quasiconformal extension to $\hat{\mathbb{C}}$ whose complex dilatation $\mu \in \mathcal{M}_0(\Delta)$; $(V_2) h \in SS;$ $(V_3) \mathcal{S}(f) \in \mathcal{B}_0(\Delta^{*}).$ We first note that the   implication relations 
$B_1 \Rightarrow B_3$ and $V_1 \Rightarrow V_3$ are still valid for the chord-arc domain.

\begin{lemma}\label{schwarz}
Let $\Omega$ be a chord-arc domain. Let $s$ be a quasiconformal mapping of  $\hat{\mathbb{C}}$ with complex dilatation 
equal to $\nu$ in $\Omega$, conformal in $\Omega^{*}$. If $|\nu|^2\rho_{\Omega}(z)dxdy \in CM(\Omega)$, then 
$|\mathcal{S}(s)|^2 \rho_{\Omega^{*}}^{-3}(z)dxdy \in CM(\Omega^{*}).$ While $|\mathcal{S}(s)|^2 \rho_{\Omega^{*}}^{-3}(z)dxdy \in CM_0(\Omega^{*}) $ if $|\nu|^2\rho_{\Omega}(z)dxdy \in CM_0(\Omega)$.
\end{lemma}

\begin{proof}
Let $g: \Delta \to \Omega $ and $f: \Delta^{*} \to \Omega^{*}$ be the two conformal mappings of a chord-arc domain. Let $h = g^{-1}\circ f$ be conformal welding with respect to $\Gamma = \partial \Omega$. Then $h \in LQS$. 

Let $\varphi = E(h)$ be the Douady-Earle extension of $h$. Douady-Earle \cite{DE} says $\varphi$ is bi-Lipschitz  under the Poincar\'e metric. It follows from Cui-Zinsmeister \cite{CZ} that $|\mu(z)|^2/(1 - |z|^2)dxdy \in CM(\Delta)$, where $\mu = \overline{\partial}\varphi/\partial \varphi$. Denote by $\tilde{f} = g \circ \varphi$. Then $\tilde{f}$ is a quasiconformal extension of $f$ to the unit disk $\Delta$ satisfying that  $\tilde{f}$ is bi-Lipschitz continuous under the Poincar\'e metric and 
$$\frac{|\mu(\tilde{f})|^2}{1 - |z|^2}dxdy = \frac{|\mu(z)|^2}{1 - |z|^2}dxdy \in CM(\Delta).$$ 
Since $\Omega$ is a chord-arc domain, by Lemma \ref{A} $g|_{\mathbb{S}}$ is a strongly quasisymmetric homeomorphism. We also have $\varphi|_{\mathbb{S}} = h \in LQS \subset SQS$, from which we deduce that $\tilde{f}|_{\mathbb{S}}$ is a strongly quasisymmetric homeomorphism by means of the similar proof as Theorem \ref{1c}.

It follows from the formula for the dilatation of a composition, 
\begin{equation*}
\mu(s \circ \tilde{f}) = \frac{\mu(\tilde{f}) + (\nu \circ \tilde{f})\tau}{1 + \overline{\mu(\tilde{f})}(\nu \circ \tilde{f})\tau}, \;\;\;\; \tau = \frac{\overline{(\tilde{f})_z}}{(\tilde{f})_z}
\end{equation*}
that
\begin{equation*}
\frac{|\mu(s \circ \tilde{f})|^2}{1 - |z|^2} \leqslant C_1 \big(\frac{|\mu(\tilde{f})|^2}{1 - |z|^2} + \frac{|\nu \circ \tilde{f}|^2}{1 - |z|^2}\big).
\end{equation*}
Noting that $\tilde{f} = g \circ \varphi$,  we see that 
\begin{equation*}
\begin{split}
\frac{|\nu \circ \tilde{f}|^2}{1 - |z|^2} & \leqslant C_2 \frac{|(\nu \circ g) \circ \varphi|^2}{1 - |\varphi(z)|^2} |\partial \varphi | \\
& = C_2 \big(\frac{|\nu \circ g|^2}{1 - |z|^2}\big) \circ \varphi |\partial \varphi|\\
& = C_2 (|\nu \circ g|^2\rho_{\Omega}\circ g |g^{'}|)\circ \varphi |\partial \varphi|\\
& = C_2 ((|\nu|^2\rho_{\Omega}) \circ g |g^{'}|)\circ \varphi |\partial \varphi|\\
& = C_2 (|\nu|^2\rho_{\Omega})\circ \tilde{f} |\partial \tilde{f}|,\\
\end{split}
\end{equation*}
which implies that $|\nu \circ \tilde{f}|^2/(1 - |z|^2)dxdy \in CM(\Delta)$ by Theorem \ref{2c}.  Thus, $|\mu(s \circ \tilde{f})|^2/(1 - |z|^2)dxdy \in CM(\Delta).$ 
We conclude by $B_1 \Rightarrow B_3$ that 
$$|\mathcal{S}(s \circ f)|^2 (|z|^2 - 1)^3dxdy \in CM(\Delta^{*}). $$
Since $\Omega^{*}$ is a chord-arc domain, we also have
$$|\mathcal{S}( f)|^2 (|z|^2 - 1)^3dxdy \in CM(\Delta^{*}).$$
By simple computation, 
\begin{equation*}
\begin{split}
(|\mathcal{S}(s)|^2 \rho_{\Omega^{*}}^{-3}) \circ f |f^{'}| & = |\mathcal{S}(s) \circ f|^2 |f^{'}|^4 (\rho_{\Omega^{*}} \circ f |f^{'}|)^{-3}\\
& = |\mathcal{S}(s) \circ f  |^2|f^{'}|^4 (|z|^2 - 1)^3 \\
& \leqslant 4 |\mathcal{S}(s \circ f)|^2 (|z|^2 - 1)^3 + 4|\mathcal{S}(f)|^2 (|z|^2 - 1)^3.\\
\end{split}
\end{equation*}
Thus, $(|\mathcal{S}(s)|^2 \rho_{\Omega^{*}}^{-3}) \circ f(z) |f^{'}|dxdy  \in CM(\Delta^{*})$. Now $|\mathcal{S}(s)|^2 \rho_{\Omega^{*}}^{-3}(z)dxdy \in CM(\Omega^{*})$ 
follows from Theorem \ref{1c}.

Examining the above proof, we may obtain that $|\mathcal{S}(s)|^2 \rho_{\Omega^{*}}^{-3}(z)dxdy \in CM_0(\Omega^{*})$ if  $|\nu|^2\rho_{\Omega}(z)dxdy \in CM_0(\Omega)$.
\end{proof}

Shen-Wei \cite{SW} proved $\varPhi$ from $M\ddot{o}b(\mathbb{S})\backslash SQS$ onto its image in $\mathcal{B}(\Delta^{*})$ is a homeomorphism if we think of 
$M\ddot{o}b(\mathbb{S})\backslash SQS$ as having the BMO topology and $\mathcal{B}(\Delta^{*})$ as having the topology induced by the Carleson norm. 
In \cite{Zi85} it is proved that $M\ddot{o}b(\mathbb{S})\backslash LQS$ is an open subset of $M\ddot{o}b(\mathbb{S})\backslash SQS$.  
We conclude  that $\varPhi$ from $M\ddot{o}b(\mathbb{S})\backslash LQS$ onto its image in 
$\mathcal{B}(\Delta^{*})$ is also a homeomorphism. Thus, the map $\varPhi$ yields a global coordinate for $M\ddot{o}b(\mathbb{S})\backslash LQS$ in the Banach space $\mathcal{B}(\Delta^{*})$, and then the coset space $M\ddot{o}b(\mathbb{S})\backslash LQS$ becomes a complex manifold modeled in the Banach space $\mathcal{B}(\Delta^{*})$.

\bigskip

In the rest part of this paper, we prove $\hat{\varPhi}$ described in Section 2 from $SS\backslash LQS$ onto its image in
$\mathcal{B}_0\backslash \mathcal{B}$ is well-defined and globally one-to-one, and then $SS\backslash LQS$ is a complex manifold. 

In the first step, we claim $\hat{\varPhi}$ is well-defined. Since an element $h \in LQS$ is determined by the complex dilatation 
$\mu \in \mathcal{M}(\Delta)$ of the quasiconformal extension of $h$ to $\Delta$, we can write $\varPhi(\mu)$ instead of $\varPhi(h)$. 

Let $\partial \Omega_1 = \partial \Omega_1^{*} = \Gamma$ be a chord-arc curve passing through three points $1, \, i$ and $-1$. 
Let $f_0$ be a conformal map of $\Delta^{*}$ to $\Omega_1^{*}$, $g_0$ be a conformal map of $\Omega_1$ to $\Delta$, and $g_0 \circ f_0$ restricted to 
$\mathbb{S}$ be equal to $h_0$. Suppose all three maps are normalized to fix $1, \, i$ and $-1$. 
Let $f^{\lambda} = DE(h_0)$ be the Douady-Earle extension of $h_0$ with Beltrami coefficient equal to $\lambda$ in $\Delta$. Denote by 
$f_{\lambda} = g_0^{-1}\circ f^{\lambda}$ in $\Delta$, and $f_{\lambda} = f_0$ in $\Delta^{*}$. Then $f_{\lambda}$ is a quasiconformal homeomorphism of 
$\hat{\mathbb{C}}$ with Beltrami coefficient equal to $\lambda$ in $\Delta$ and  equal to $0$ in $\Delta^{*}$.

In order to prove the claim, we just need to show: for any $\mu \in \mathcal{M}_0(\Delta)$, $\Phi(\mu * \lambda) - \Phi(\lambda) \in \mathcal{B}_0(\Delta^{*})$.

Let $s^{\mu}$ be a quasiconformal mapping of $\Delta$ onto $\Delta$ with Beltrami coefficient $\mu$. Let $\tilde{s}$ be the quasiconformal homeomorphism 
of $\hat{\mathbb{C}}$ with the property that $\tilde{s}\circ f_{\lambda}$ has the same Beltrami coefficient as $s^{\mu}\circ f^{\lambda}$ in $\Delta$ and 
$\tilde{s}\circ f_{\lambda}$ has Beltrami coefficient identically equal to zero in $\Delta^{*}$. If $r$ is the conformal mapping of 
$\tilde{s}\circ f_{\lambda}(\Delta)$ onto $\Delta$, then 
$$s^{\mu}\circ f^{\lambda} = r \circ \tilde{s} \circ f_{\lambda} =  r \circ \tilde{s} \circ g_0^{-1} \circ g_0 \circ f_{\lambda}  =    r \circ \tilde{s} \circ g_0^{-1} \circ f^{\lambda}.$$
Cancelling $f^{\lambda}$ from the right and left side of this equation, we find that 
$$s^{\mu}\circ g_0 = r \circ \tilde{s}.$$
Thus, we conclude that the Beltrami coefficient of $\tilde{s}$ is 
$$\tilde{\nu}(z) = \mu \circ g_0(z) \frac{\overline{g_0^{'}(z)}}{g_0^{'}(z)}.$$

Now, suppose $\mu \in \mathcal{M}_0(\Delta)$. Then 
$$d\lambda_{\mu}(z) = \frac{|\mu(z)|^2}{1 - |z|^2}dxdy \in CM_0(\Delta).$$
Since $\Omega_1$ is a chord-arc domain, we conclude by Theorem \ref{Omega} that
$$|\tilde{\nu}(z)|^2\rho_{\Omega_1}(z)dxdy = \frac{|\mu \circ g_0(z)|^2}{1 - |g_0(z)|^2}|g_0^{'}(z)|dxdy = \lambda_{\mu}\circ g_0(z) |g_0^{'}(z)|dxdy \in CM_0(\Omega_1).$$
It follows from Lemma \ref{schwarz} that 
$$|\mathcal{S}(\tilde{s})|^2\rho_{\Omega_1^{*}}^{-3}(z)dxdy \in CM_0(\Omega_1^{*}).$$
By means of Theorem \ref{Delta} we have
$$|\mathcal{S}(\tilde{s})\circ f_{\lambda}(z)|^2\rho_{\Omega_1^{*}}^{-3}(f_{\lambda}(z))|f_{\lambda}^{'}|dxdy \in CM_0(\Delta^{*}).$$
The cocycle identity for the Schwarzian derivative implies 
$$\Phi(\mu * \lambda) - \Phi(\lambda) = \mathcal{S}(\tilde{s}\circ f_{\lambda}) - \mathcal{S}(f_{\lambda}) = 
\mathcal{S}(\tilde{s})\circ f_{\lambda}(f_{\lambda}^{'})^2.$$
On the other hand, $$|\mathcal{S}(\tilde{s})\circ f_{\lambda}(f_{\lambda}^{'})^2|^2 (|z|^2 - 1)^3dxdy = 
(|\mathcal{S}(\tilde{s})|^2\rho_{\Omega_1^{*}}^{-3})\circ f_{\lambda} |f_{\lambda}^{'}|dxdy \in CM_0(\Delta^{*}).$$
Consequently, $\Phi(\mu * \lambda) - \Phi(\lambda) \in \mathcal{B}_0(\Delta^{*})$.

Thus, the claim is proved. We have the commutative diagram:

\begin{tikzpicture}
\matrix(m)[matrix of math nodes, row sep=5em, column sep=5em]
{ |[name=ka]| M\ddot{o}b(\mathbb{S})\backslash LQS & |[name=kb]| \mathcal{B}\\
|[name=kc]| SS\backslash LQS & |[name=kd]| \mathcal{B}_0\backslash \mathcal{B}\\};

\draw[->, thick] (ka) edge node[auto] {$\varPhi$} (kb)
                         (ka) edge node[auto] {$\pi$} (kc)
                         (kc) edge node[auto] {$\hat{\varPhi}$} (kd)
                         (kb) edge node[auto] {$p$} (kd)
                         ;
\end{tikzpicture}

In the second step, we prove $\hat{\varPhi}$ is locally one-to-one.  Assume $h$ be a fixed point in $LQS$. 
Pick a small neighborhood of $h$ in $LQS$ which by $\pi$ is mapped onto a neighborhood of $\pi(h)$ in $SS\backslash LQS$. Let $h_0$ and $h_1$ be two functions in this neighborhood, and $h_0 = g_0\circ f_0$, $h_1 = g_1\circ f_1$ be the Riemann factorizations of these two mappings. Suppose all these maps are normalized to fix $1, \, i$ and $-1$. In order to prove $\hat{\varPhi}$ is locally one-to-one, we need to show, if $\mathcal{S}(f_1) - \mathcal{S}(f_0) \in \mathcal{B}_0$, then there exists a 
$s \in SS$ such that $s \circ h_0 = h_1$.

According to the assumption that two functions $h_0$ and $h_1$ are in a small neighborhood of $h$ in $LQS$, there is a small constant $\epsilon > 0$  such that 
$\|\mathcal{S}(f_1) - \mathcal{S}(f_0)\|_{\mathcal{B}} < \epsilon.$ Since the inclusion map $i: \mathcal{B}(\Delta^{*}) \to B(\Delta^{*})$ is continuous, 
$\|\mathcal{S}(f_1) - \mathcal{S}(f_0)\|_{B} < C\epsilon.$  Let $f = f_1\circ f_0^{-1}$ be the conformal map of 
$\Omega_1^{*} = f_0(\Delta^{*})$. Denote by $\rho_{\Omega_1^{*}}$ and $\rho_{\Delta^{*}}$ the hyperbolic metric in $\Omega_1^{*}$ and $\Delta^{*}$, respectively, that is, 
$\rho_{\Omega_1^{*}} = \rho_{\Delta^{*}} \circ f_0^{-1}|(f_0^{-1})^{'}|$.  The cocycle identity for the Schwarzian derivative implies 
\begin{equation*}
\begin{split}
|\mathcal{S}(f)| \rho_{\Omega_1^{*}}^{-2} & = |(\mathcal{S}(f_1) - \mathcal{S}(f_0))\circ f_0^{-1}(f_0^{-1})^{'2}|(\rho_{\Delta^{*}}\circ f_0^{-1} |(f_0^{-1})^{'}|)^{-2}\\
& = (|\mathcal{S}(f_1) - \mathcal{S}(f_0)|\rho_{\Delta^{*}}^{-2})\circ f_0^{-1}.\\
\end{split}
\end{equation*}
Thus, $\sup_{z \in \Omega_1^{*}}|\mathcal{S}(f)| \rho_{\Omega_1^{*}}^{-2} = \|\mathcal{S}(f_1) - \mathcal{S}(f_0)\|_{B} < C\epsilon.$

Let $f^{\nu} = DE(h_0)$ be Douady-Earle extension of $h_0$ with complex dilatation equal to $\nu$ in $\Delta$. Then $f_{\nu} = g_0^{-1}\circ f^{\nu}$ is a quasiconformal extension of $f_0$ to $\hat{\mathbb{C}}$. The Earle-Nag reflection \cite[p.263]{GL} associated with the curve $\Gamma_1 = \partial \Omega_1$ is given by the formula 
$$
\gamma(z) = \begin{cases}
f_0 \circ j \circ f_{\nu}^{-1}(z) = f_0 \circ j \circ DE(g_0 \circ f_0)^{-1}\circ g_0, &  z \in \Omega_1\\
z, & z \in \Gamma_1\\
\gamma^{-1}(z), & z \in \Omega_1^{*}
\end{cases}
$$
where $j(z) = 1/\bar{z}$, and \cite[p.265]{GL}  says 
\begin{equation}\label{gamma}
C_1^{-1}(\|\nu\|_{\infty}) \leqslant |\gamma(z) - z|^2 \rho_{\Omega_1^{*}}^{-2} (\gamma(z)) |\bar{\partial} \gamma(z)| \leqslant C_1(\|\nu\|_{\infty}).
\end{equation}
Under the condition that $\sup_{z \in \Omega_1^{*}}|\mathcal{S}(f)| \rho_{\Omega_1^{*}}^{-2}$ is sufficiently small,   Ahlfors \cite{Ah66}, Earle-Nag \cite{EN} (see also \cite[p.266]{GL}) proved that $f$ can be extended to a quasiconformal mapping $f_{\mu}$ in $\Omega_1$ whose complex dilatation $\mu$ satisfies
\begin{equation*}
\mu(z) = \frac{\mathcal{S}(f)(\gamma(z))(\gamma(z) - z)^2 \bar{\partial} \gamma(z)}{2 + \mathcal{S}(f)(\gamma(z))(\gamma(z) - z)^2 \partial \gamma(z)}, \;\;\; z \in \Omega_1.
\end{equation*}
Then by means of (\ref{gamma}) we have
\begin{equation*}
|\mu(z)| \leqslant C_2(\|\nu\|_{\infty}) |\mathcal{S}(f)(\gamma(z))|\rho_{\Omega_1^{*}}^{-2} (\gamma(z)), \;\;\; z \in \Omega_1.
\end{equation*}

Note that on the unit circle $h_1 \circ h_0^{-1} = (g_1 \circ f_1)\circ (g_0 \circ f_0)^{-1} = g_1 \circ f \circ g_0^{-1}$ and in the unit disk $g_1 \circ f_{\mu} \circ g_0^{-1}$ is an extension of $h_1 \circ h_0^{-1}$, whose complex dilatation has the form
$$\mu \circ g_0^{-1} \frac{\overline{(g_0^{-1})^{'}}}{(g_0^{-1})^{'}}.$$
To prove $h_1 \circ h_0^{-1} \in SS$, it is sufficiently to show that 
$$\frac{|\mu \circ g_0^{-1}(z)|^2}{1 - |z|^2}dxdy \in CM_0(\Delta).$$
By simple computation, 
\begin{equation*}
\begin{split}
|\mu \circ g_0^{-1}| & \leqslant C_2(\|\nu\|_{\infty}) |\mathcal{S}(f)(f_0 \circ j \circ f_{\nu}^{-1}\circ g_0^{-1})|\rho_{\Omega_1^{*}}^{-2} (f_0 \circ j \circ f_{\nu}^{-1}\circ g_0^{-1})\\
& = C_2(\|\nu\|_{\infty}) |\mathcal{S}(f)(f_0 \circ j \circ (f^{\nu})^{-1})|\rho_{\Omega_1^{*}}^{-2} (f_0 \circ j \circ (f^{\nu})^{-1}).\\
\end{split}
\end{equation*}
Then,
\begin{equation*}
\begin{split}
|\mu \circ g_0^{-1}\circ f^{\nu}| & \leqslant C_2(\|\nu\|_{\infty}) |\mathcal{S}(f)(f_0 \circ j )|\rho_{\Omega_1^{*}}^{-2} (f_0 \circ j )\\
& =   C_2(\|\nu\|_{\infty}) |(\mathcal{S}(f_1) - \mathcal{S}(f_0))(j(z)) ((f_0^{'})(j(z)))^{-2}| (|j(z)|^2 - 1)^2 |(f_0^{'})(j(z))|^2\\
& =   C_2(\|\nu\|_{\infty}) |(\mathcal{S}(f_1) - \mathcal{S}(f_0))(j(z))| (|j(z)|^2 - 1)^2.\\
\end{split}
\end{equation*}
Consequently, by $\mathcal{S}(f_1) - \mathcal{S}(f_0) \in \mathcal{B}_0$ again and
\begin{equation*}
\begin{split}
\Big(\frac{|\mu \circ g_0^{-1} \circ f^{\nu}(\bar{t})|^2}{1 - |\bar{t}|^2}\Big)& \circ \frac{1}{w}\frac{1}{|w|^2}  = 
\frac{|\mu \circ g_0^{-1} \circ f^{\nu}(\frac{1}{\bar{w}})|^2}{|w|^2 - 1} = \frac{|\mu \circ g_0^{-1} \circ f^{\nu}|^2}{1 - |z|^2}\\
& \leqslant C_2^2(\|\nu\|_{\infty})|(\mathcal{S}(f_1) - \mathcal{S}(f_0))(w)|^2 (|w|^2 - 1)^3,\\
\end{split}
\end{equation*}
where $w = \frac{1}{\bar{z}} \in \Delta^{*},$  we have 
$$\frac{|\mu \circ g_0^{-1} \circ f^{\nu}(z)|^2}{1 - |z|^2}dxdy \in CM_0(\Delta).$$
It follows from bi-Lipschitz continuity of the Douady-Earle extension $f^{\nu}$ under the Poincar\'e metric that 
\begin{equation*}
\begin{split}
\frac{|\mu \circ g_0^{-1}(z)|^2}{1 - |z|^2}dxdy & \leqslant C_3(\|\nu\|_{\infty})\frac{|\mu \circ g_0^{-1}|^2}{1 - |z|^2}
\frac{1 - |z|^2}{1 - |(f^{\nu})^{-1}|^2}|\partial((f^{\nu})^{-1})|dxdy\\
& = C_3(\|\nu\|_{\infty})\Big(\frac{|\mu \circ g_0^{-1} \circ f^{\nu}|^2}{1 - |z|^2}\Big)\circ (f^{\nu})^{-1}|\partial((f^{\nu})^{-1})|dxdy \in CM_0(\Delta).\\
\end{split}
\end{equation*}
We conclude that $h_1 \circ h_0^{-1} \in SS$.\\

Finally, we claim that the local injectivity can be improved as in \cite[p.320]{GL} to be globally injective.  Here is an outline of the proof following the idea of Jeremy Kahn used for the analogous statement in \cite{GL}.

For convenience we switch to the half-plane. The setting is then two pairs of normalized welding maps $f_j,g_j$ for a normalized elements of $SQS,\,h_j=g_j\circ f_j$, with $g_j^{-1}$ and $f_j$ mapping respectively the upper and lower half-plane to the two complementary components of a chord-arc curve $\Gamma_j.$ We assume that $S(f_1)-S(f_0)$ leads to a small Carleson measure and we want to prove that it follows that $log h'\in VMO(\R)$, where $h=h_1\circ h_0^{-1}$.  Let $R_0$ be a square in the lower half-plane whose base is a small interval $I$ of the real axis. Its image $R$  by $f_0$ is a chord-arc domain and the restriction on $R$ of the Schwarzian derivative of $f=f_1\circ f_0^{-1}$ leads, via the Bers embedding, to a Carleson measure with small norm. Therefore, using Earle-Nag reflection, it can be shown that $f=f_1\circ f_0^{-1}$ extends to a global quasiconformal map $F$ whose dilatation leads to a Carleson measure with small norm wrt the complement of $R$. Pulling back this information via $f_0$ one then finds that $h=h_1\circ h_0^{-1}$ extends to a quasiconformal mapping $G$ of a domain of the form $\Omega_0=R_0\cup\R_0^*$ where $\R_0^*$ is a chord-arc domain included in the upper-half-plane, sending $\Omega_0\cup\R$ into $\R$,  such that $\partial \Omega_0\cap \R=\partial R_0\cap \R$, and whose dilatation leads to a Carleson measure with small norm.

Such a map can easily be extended to a  quasiconformal self-mapping of the upper half plane whose dilatation leads to a Carleson measure with small norm. We then invoke \cite{FKP}: $log(G')\vert I \vert =log h'$ has a small $BMO$ norm which is controlled by $\epsilon(\vert I\vert)$ with $\epsilon(t)\to 0$ as $t\to 0$. We conclude that $log h'$ must be in $VMO$, the fact we intended to prove.

\section{Appendix: another proof of global injectivity}

After finishing all above sections, we find, following  the idea of Matsuzaki in \cite{Ma07}, the global injectivity of the map $\hat{\varPhi}$ 
from $SS\backslash LQS$ into $\mathcal{B}_0\backslash \mathcal{B}$ can be obtained in a totally different way. The strategy is simple and can be explained as follows. Assume we have an asymptotically conformal homeomorphism $f$ in the situation we consider, then we can decompose $f$ into two quasiconformal homeomorphisms. One is within the neighborhood where the local injectivity can be applied (the local injectivity has been proved before), and the other is asymptotically conformal whose support of the complex dilatation is contained in a compact subset. We easily see that the latter mapping comes from the trivial coset. The argument for the rigorous proof is as follows. 

Let $h_0 = g_0\circ f_0$ and $h_1 = g_1\circ f_1$ be welding homeomorphisms corresponding to chord-arc curves $\Gamma = \partial \Omega = \partial \Omega^{*}$ and 
 $\Gamma_1 = \partial\Omega_1 = \partial\Omega_1^{*}$, respectively. Suppose all these maps are normalized to fix $1, i$ and $-1$. We will show, if 
 $\mathcal{S}(f_1) - \mathcal{S}(f_0) \in \mathcal{B}_0$ then $h_1\circ h_0^{-1} \in SS.$
 
 We adopt the same notations as before. Let $f^{\lambda} = DE(h_0)$ be the Douady-Earle extension of $h_0$ with complex dilatation equal to $\lambda$ in $\Delta$. Then 
 $f_{\lambda} = g_0^{-1}\circ f^{\lambda}$ is an extension of $f_0$ to $\hat{\mathbb{C}}$.  Suppose $\mathcal{S}(f_1) - \mathcal{S}(f_0) = \varphi \in \mathcal{B}_0 \subset B_0$. Let the subset $M_0(\Delta)$ of $M(\Delta)$ consist of all Beltrami coefficients vanishing at the boundary. 
 It follows from Theorem \ref{2.1} (see also \cite[Theorem 3.6]{Ma}) that there exists a Beltrami coefficient $\mu \in M_0(\Delta)$ such that 
 $\varPhi(\mu * \lambda) = \mathcal{S}(f_1)$, which says $f^{\mu}\circ f^{\lambda}$ is an extension of $h_1$ to the unit disk $\Delta$. Thus 
 $g_1^{-1}\circ f^{\mu}\circ f^{\lambda}$ is a quasiconformal extension of $f_1$ to $\hat{\mathbb{C}}$. Let $\hat{f}$ be a quasiconformal homeomorphism of the whole plane          $\hat{\mathbb{C}}$ equal to $g_1^{-1}\circ f^{\mu}\circ f^{\lambda}\circ f_{\lambda}^{-1} = g_1^{-1}\circ f^{\mu}\circ g_0$ in $\Omega$ and equal to $f_1\circ f_0^{-1}$ in 
 $\Omega^{*}$. Then the complex dilatation $\hat{\mu}$ of $\hat{f}$ vanishes at the boundary $\Gamma$. In particular, for any $\epsilon > 0$, we can choose a compact subset $\Omega_0 \subset \Omega$ such that 
 $$12\|\hat{\mu}|_{\Omega - \Omega_0}\|_{\infty} \leqslant \epsilon.$$
 We decompose $\hat{f}$ into $\hat{f_0}\circ \hat{f_1}$ as follows. The quasiconformal homeomorphism $\hat{f_1}: \hat{\mathbb{C}} \to \hat{\mathbb{C}}$ is chosen so that its complex dilatation coincides with $\hat{\mu}$ on $\Omega - \Omega_0$ and zero elsewhere. Then $\hat{f_0}$ is defined to be $\hat{f}\circ \hat{f_1}^{-1}$. Thus the complex dilatation $\mu(\hat{f_0})$ of $\hat{f_0}$ is zero in $\Omega_2 - \hat{f_1}(\Omega_0)$. Here $\Omega_2 = \hat{f_1}(\Omega)$. So $\mu(\hat{f_0})$ induces a vanishing Carleson measure in $\Omega_2$. It follows from Lemma \ref{schwarz} that $|\mathcal{S}(\hat{f_0})|^2\rho_{\Omega_2^{*}}^{-3}dxdy \in CM_0(\Omega_2^{*})$, and then 
 $$|\mathcal{S}(\hat{f_0})\circ (\hat{f_1}\circ f_0)(\hat{f_1}\circ f_0)^{'2}|^2\rho_{\Delta^{*}}^{-3}dxdy = (|\mathcal{S}(\hat{f_0})|^2\rho_{\Omega_2^{*}}^{-3})\circ 
 (\hat{f_1}\circ f_0)|(\hat{f_1}\circ f_0)^{'}|dxdy \in CM_0(\Delta^{*}),$$
 which implies $\mathcal{S}(\hat{f_0})\circ (\hat{f_1}\circ f_0)(\hat{f_1}\circ f_0)^{'2} \in \mathcal{B}_0(\Delta^{*})$. Thus, we have 
 \begin{equation}\label{mathcalB0}
 \begin{split}
 \mathcal{S}(\hat{f_1}\circ f_0) - \mathcal{S}(f_0) & = \mathcal{S}(\hat{f_0}\circ \hat{f_1}\circ f_0) - \mathcal{S}(f_0) - \mathcal{S}(\hat{f_0})\circ (\hat{f_1}\circ f_0)(\hat{f_1}\circ f_0)^{'2}\\
 & = \mathcal{S}(f_1) - \mathcal{S}(f_0) - \mathcal{S}(\hat{f_0})\circ (\hat{f_1}\circ f_0)(\hat{f_1}\circ f_0)^{'2} \in \mathcal{B}_0(\Delta^{*}).\\
\end{split}
\end{equation}

Let $\hat{g_1}: \Omega_2 \to \Delta$ be conformal and $\hat{h_1} = \hat{g_1}\circ \hat{f_1}\circ f_0$  be the welding homeomorphism corresponding to the chord-arc curve 
$\Gamma_2 = \partial \Omega_2$. Then we have
\begin{equation*}
 \begin{split}
h_1 = g_1 \circ f_1 & = g_1\circ \hat{f}\circ f_0 = g_1\circ \hat{f_0}\circ \hat{f_1}\circ f_0 \\
& = g_1\circ \hat{f_0}\circ \hat{g_1}^{-1} \circ \hat{g_1} \circ \hat{f_1}\circ f_0 = g_1\circ \hat{f_0}\circ \hat{g_1}^{-1} \circ \hat{h_1}.\\
\end{split}
\end{equation*}
The complex dilatation of $g_1\circ \hat{f_0}\circ \hat{g_1}^{-1}$ induces a vanishing Carleson measure in the unit disk $\Delta$. Then we have 
$g_1\circ \hat{f_0}\circ \hat{g_1}^{-1} \in SS$ in the unit circle $\mathbb{S}$. Thus, in order to prove $h_1\circ h_0^{-1} \in SS,$ we just need to show 
$\hat{h_1}\circ h_0^{-1} \in SS$. 

By simple computation, $|\mathcal{S}(\hat{f_1}\circ f_0)(z) - \mathcal{S}(f_0)(z)|\rho_{\Delta^{*}}^{-2}(z) = |\mathcal{S}(\hat{f_1})(\zeta)|\rho_{\Omega^{*}}^{-2}(\zeta)$ for 
$\zeta = f_0(z)$ and this is bounded by $12\|\hat{\mu}_{\Omega - \Omega_0}\|_{\infty}$ (see \cite[p.72]{Le}). Then 
\begin{equation}\label{B0}
\|\mathcal{S}(\hat{f_1}\circ f_0) - \mathcal{S}(f_0)\|_B \leqslant \epsilon.
\end{equation}
Combining (\ref{mathcalB0}), (\ref{B0}) and the local injectivity claim of the map $\hat{\varPhi}$, we obtain $\hat{h_1}\circ h_0^{-1} \in SS$ and then the global injectivity of $\hat{\varPhi}$ from  $SS\backslash LQS$ into $\mathcal{B}_0\backslash \mathcal{B}$.

\begin{center}
\begin{minipage}{140mm}
{\small{\bf Acknowledgement}.
The authors would like to thank the referee for a very careful reading of the manuscript and for sentence-by-sentence corrections.
}

\end{minipage}
\end{center}

\end{document}